\documentclass{article}%
\usepackage{graphicx}
\usepackage{amsmath}
\usepackage{amscd}
\usepackage{amsfonts}
\usepackage{amssymb}
\newtheorem{theorem}{Theorem}[section]

\newtheorem{proposition}[theorem]{Proposition}

\newtheorem{definition}[theorem]{Definition}

\newtheorem{remark}[theorem]{Remark}

\newenvironment{proof}[1][Proof]{\textbf{#1.} }{\ \rule{0.5em}{0.5em}}

\newcommand{\be}{\begin{equation}}
\newcommand{\ee}{\end{equation}}
\newcommand{\bes}{\begin{equation*}}
\newcommand{\ees}{\end{equation*}}

\newcommand{\cH}{\mathcal{H}}
\newcommand{\cK}{\mathcal{K}}

\newcommand{\cS}{\mathcal{S}}

\newcommand{\Rp}{\mathbb{R}_+}

\begin{document}
\title{Representing a product system representation as a contractive
semigroup and applications to regular isometric dilations}

\author{Orr Moshe Shalit}
\maketitle

\begin{abstract}
{In this paper we propose a new technical tool for analyzing
representations of Hilbert $C^*$-product systems. Using this tool,
we give a new proof that every doubly commuting representation
over $\mathbb{N}^k$ has a regular isometric dilation, and we also
prove sufficient conditions for the existence of a regular
isometric dilation of representations over more general
subsemigroups of $\Rp^k$.}
\end{abstract}

\section{Introduction, preliminaries, notation}
\subsection{Background. Correspondences, product systems and representations}
In the following paragraphs we review the definitions of our main objects of study. The reader familiar with $C^*$-correspondences, product systems of correspondences and representations of product systems,
may skip to subsection \ref{subsec:about}.

\begin{definition}
Let $A$ be a $C^*$ algebra. A \emph{Hilbert $C^*$-correspondences over $A$} is a (right) Hilbert $A$-module $E$
which carries an adjointable, left action of $A$.
\end{definition}
The following notion of representation of a $C^*$-correspondence was studied extensively in \cite{MS98}, and turned out to be a very useful tool.
\begin{definition}
Let $E$ be a $C^*$-correspondence over $A$, and let $H$ be a
Hilbert space. A pair $(\sigma, T)$ is called a \emph{completely
contractive covariant representation} of $E$ on $H$ (or, for
brevity, a \emph{c.c. representation}) if
\begin{enumerate}
    \item $T: E \rightarrow B(H)$ is a completely contractive linear map;
    \item $\sigma : A \rightarrow B(H)$ is a nondegenerate $*$-homomorphism; and
    \item $T(xa) = T(x) \sigma(a)$ and $T(a\cdot x) = \sigma(a) T(x)$ for all $x \in E$ and  all $a \in A$.
\end{enumerate}
\end{definition}
Given a $C^*$-correspondence $E$ and a c.c. representation $(\sigma,T)$ of $E$ on $H$,
one can form the Hilbert space $E \otimes_\sigma H$, which is defined as the Hausdorff completion of the algebraic tensor product with respect to the inner product
$$\langle x \otimes h, y \otimes g \rangle = \langle h, \sigma (\langle x,y\rangle) g \rangle .$$
One then defines $\tilde{T} : E \otimes_\sigma H \rightarrow H$ by
$$\tilde{T} (x \otimes h) = T(x)h .$$

As in the theory of contractions on a Hilbert space, there are certain particularly well behaved representations
which deserve to be singled out.
\begin{definition}
A c.c. representation $(T,\sigma)$ is called \emph{isometric} if for all $x, y \in E$,
\begin{equation*}
T(x)^*T(y) = \sigma(\langle x, y \rangle) .
\end{equation*}
(This is the case if and only if $\tilde{T}$ is an isometry.) It is called \emph{fully coisometric} if $\tilde{T}$ is a coisometry.
\end{definition}

Given two Hilbert $C^*$-correspondences $E$ and $F$ over $A$, the \emph{balanced} (or \emph{inner})
tensor product $E \otimes_A F$ is a Hilbert $C^*$-correpondence over $A$ defined to be the Hausdorf completion of the algebraic tensor product
with respect to the inner product
$$\langle x \otimes y, w \otimes z \rangle = \langle y , \langle x,w\rangle \cdot z \rangle \,\, , \,\,  x,w\in E, y,z\in F .$$
The left and right actions are defined as $a \cdot (x \otimes y) = (a\cdot x) \otimes y$ and $(x \otimes y)a = x \otimes (ya)$, respectively, for all $a\in A, x\in E, y\in F$.
We shall usually omit the subscript $A$, writing just $E \otimes F$.

Suppose $\cS$ is an abelian cancellative semigroup with identity $0$ and $p : X \rightarrow \cS$ is a family of $C^*$-correspondences over $A$. Write $X(s)$ for the correspondence $p^{-1}(s)$ for $s \in \cS$. We say that $X$ is a (discrete) \emph{product system} over $\cS$ if $X$ is a semigroup, $p$ is a semigroup homomorphism and, for each $s,t \in \cS \setminus \{0\}$, the map $X(s) \times X(t) \ni (x,y) \mapsto xy \in X(s+t)$ extends to an isomorphism $U_{s,t}$ of correspondences from $X(s) \otimes_A X(t)$ onto $X(s+t)$. The associativity of the multiplication means that, for every $s,t,r \in \cS$,
\begin{equation}\label{eq:assoc_prod}
U_{s+t,r} \left(U_{s,t} \otimes I_{X(r)} \right) = U_{s,t+r} \left(I_{X(s)} \otimes U_{t,r} \right).
\end{equation}
We also require that $X(0) = A$ and that the multiplications $X(0) \times X(s) \rightarrow X(s)$ and $X(s) \times X(0) \rightarrow X(s)$ are given by the left and right actions of $A$ and $X(s)$.

\begin{definition}
Let $H$ be a Hilbert space, $A$ a $C^*$-algebra and $X$ a product system of Hilbert $A$-correspondences over the semigroup $\cS$. Assume that $T : X \rightarrow B(H)$, and write $T_s$ for the restriction of $T$ to $X(s)$, $s \in \cS$, and $\sigma$ for $T_0$. $T$ (or $(\sigma, T)$) is said to be a \emph{completely contractive covariant representation} of $X$ if
\begin{enumerate}
    \item For each $s \in \cS$, $(\sigma, T_s)$ is a c.c. representation of $X(s)$; and
    \item $T(xy) = T(x)T(y)$ for all $x, y \in X$.
\end{enumerate}
$T$ is said to be an isometric (fully coisometric) representation if it is an isometric (fully coisometric) representation on every fiber $X(s)$.
\end{definition}
Since we shall not be concerned with any other kind of representation, we shall call a completely contractive covariant representation of a product system simply a \emph{representation}.

\subsection{What this paper is about}\label{subsec:about}
In many ways, representations of product systems are analogous to semigroups of contractions on Hilbert spaces. Indeed, given
a representation $(\sigma,T)$ of a correspondence $E$ (over a $C^*$-algebra $A$), the map $\tilde{T}$ associated with the representation is ``just"
a contraction between Hilbert spaces. When $A = E = \mathbb{C}$, then a $T$ itself \emph{is} a contraction (to see the connection with semigroups in this trivial example, note that every contraction $W$ on a Hilbert space gives rise
to semigroup of contractions $\{W^n\}_{n\in\mathbb{N}}$). Furthermore, many proofs of results concerning
representations
are based on the ideas of the proofs of the analogous results concerning contractions on a Hilbert space,
with the appropriate, sometimes highly non-trivial, modifications made. For example, the proof given in
\cite{MS02} that every representation has an isometric dilation uses some methods from the classical proof that every contraction on a Hilbert space has an isometric dilation.

The point of view we adopt in this paper is that one may try to
exploit the \emph{results} rather than the  \emph{methods} of the
theory of contractive semigroups on a Hilbert space when attacking
problems concerning representations of product systems. In other
words, we wish to find a systematic way to \emph{reduce} (problems
concerning) a representation of a product system to (analagous
problems concerning) a \emph{semigroup of contractions on a
Hilbert space}. This paper contains, we would like to think, a
first step in this direction. In section \ref{sec:rep}, given a
product system $X$ over a semigroup $\cS$ and representation
$(\sigma,T)$ of $X$ on a Hilbert space $H$, we construct a Hilbert
space $\cH$ and a contractive semigroup $\hat{T} =
\{\hat{T}_s\}_{s\in\cS}$ on $\cH$, such that $\hat{T}$ contains
all the information regarding the representation (except $\sigma =
T_0$, which takes part in the construction of $\cH$). In section
\ref{sec:dil} we show that if $\hat{T}$ has a regular isometric
dilation, then so does $T$.

In section \ref{sec:dbly}, we prove that doubly commuting
representations of product systems of Hilbert correspondences over
certain subsemigroups of $\Rp^k$ have doubly commuting, regular
isometric dilations. This was already proved in \cite{S06} for the
case $\cS = \mathbb{N}^k$. Our proof is based on the construction
made in section \ref{sec:rep}.

This is a good point to remark that our approach has some
limitations. For example, the construction introduced in section
\ref{sec:rep} does not seem to be canonical in any nice way. Also,
we cannot, using the method introduced here, obtain all of the
results in \cite{S06}. We will illustrate these limitations in
section \ref{sec:further}, after proving another sufficient
condition for the existence of a regular, isometric dilation. One
might wonder, indeed, how far can one get by trying to reduce
representations of product systems to semigroups of operators on a
Hilbert space, as the former are certainly ``much more
complicated". In this context, let us just mention that in another
paper (\cite{Shalit07a}), we will show how we can obtain by these
methods another result that has not yet been proved by other
means, namely the existence of an isometric dilation to a
\emph{fully-coisometric} representation of product systems over a
subsemigroup of $\Rp^k$.

\subsection{Notation}
A \emph{commensurable semigroup} is a semigroup $\Sigma$ such that for every $N$
elements $s_1, \ldots, s_N \in \Sigma$, there exist $s_0 \in \Sigma$ and $a_1, \ldots, a_N \in \mathbb{N}$ such that
$s_i = a_i s_0$ for all $i = 1, \ldots N$. For example, $\mathbb{N}$ is a commensurable semigroup. If $r\in \Rp$, then $r\cdot \mathbb{Q}_+$ is commensurable, and any commensurable subsemigroup of $\Rp$ is contained in such a semigroup.

Throughout this paper, $\Omega$ will denote some fixed set, and $\cS$ will denote the semigroup
$$\cS = \sum_{i\in \Omega}\cS_i ,$$
where $\cS_i$ is a commensurable and unital (i.e., contains $0$) subsemigroup of $\Rp$.
To be more precise,
$\cS$ is the subsemigroup of $\Rp^\Omega$ of finitely supported functions $s$ such that $s(j) \in \cS_j$ for all $j \in \Omega$. Still another way to describe $\cS$ is the following:
\bes
\cS = \left\{ \sum_{j\in \Omega} {\bf e_j}(s_j) : s_j \in \cS_j, {\rm \,\,all \,\, but\,\, finitely\,\, many\,} s_j {\rm's \,\,are\, }0 \right\},
\ees
where ${\bf e_i}$ is the inclusion of $\cS_i$ into $\prod_{j\in \Omega}\cS_j$. Here is a good example to keep in mind: if $|\Omega| = k \in \mathbb{N}$, and if $\cS_i = \mathbb{N}$ for all $i\in\Omega$, then $\cS = \mathbb{N}^k$.
We denote by $\cS - \cS$ the subgroup of $\mathbb{R}^\Omega$ generated by $\cS$ (with addition and subtraction defined in the obvious way). For $s \in \cS - \cS$ we shall denote by $s_+$ the element in $\cS$ that sends $j\in \Omega$ to $\max\{0,s(j)\}$, and $s_- = s_+ - s$. It is worth noting that $s\in \cS-\cS$, then $s_+$ and $s_-$ are both in $\cS$.

$\cS$ becomes a partially ordered set if one introduces the relation
$$s \leq t \Longleftrightarrow s(j) \leq t(j) \,\, , \,\, j\in\Omega .$$
The symbols $<$, $\ngeq$, etc., are to be interpreted in the obvious way.

If $u = \{u_1, \ldots, u_N\} \subseteq \Omega$, we let $|u|$ denote the number of elements in $u$ (this notation will only be used for finite sets). We shall denote by ${\bf e}[u]$ the element of $\mathbb{R}^\Omega$ having $1$ in the $i$th place for every $i\in u$, and having $0$'s elsewhere, and we denote $s[u]: = {\bf e}[u]\cdot s$, where multiplication is pointwise.

The reader might note that the constructions made in the next
section make sense for (slightly) more general semigroups, but we
shall exploit this construction in sections \ref{sec:dbly} and
\ref{sec:further} only for the semigroup $\cS$.

\section{Representing representations as contractive semigroups on a Hilbert space}\label{sec:rep}
In this section we describe the main issue of this paper -- the representation of a product system representation
as a semigroup of contractions on a Hilbert space.

For the time being, we can replace $\cS$ by any abelian cancellative semigroup with identity $0$ and an appropriate partial ordering (for example, $\cS$ can be taken to be $\mathbb{R}_+^k$). We shall intentionally avoid making our statements in the most general form in order to avoid technicalities.

Let $A$ be a $C^*$-algebra, and let $X$ be a discrete product
system of $C^*$-correspondences over $\cS$. Let $(\sigma,T)$ be a
completely contractive covariant representation of $X$ on the
Hilbert space $H$. Our assumptions do not imply that $X(0) \otimes
H \cong H$. This unfortunate fact will not cause any real trouble,
but it will make our exposition a little clumsy.

Define $\cH_0$ to be the space of all finitely supported functions
$f$ on $\cS$ such that for all $0 \neq s \in \cS$, $f(s) \in X(s)
\otimes_{\sigma} H$ and such that $f(0)\in H$. We equip $\cH_0$
with the inner product
$$\langle \delta_s \cdot \xi, \delta_t \cdot \eta \rangle = \delta_{s,t} \langle \xi, \eta \rangle  ,$$
for all $s,t \in \cS - \{0\}, \xi \in X(s) \otimes H, \eta \in
X(t) \otimes H$ (where the $\delta$'s on the left hand side are
Dirac deltas, the $\delta$ on the right hand side is Kronecker's
delta). If one of $s$ or $t$ is $0$, then the inner product is
defined similarly. Let $\cH$ be the completion of $\cH_0$ with
respect to this inner product. Note that
$$\cH \cong H \oplus \Big(\oplus_{0 \neq s \in \cS} X(s)\otimes H \Big),$$
but defining it as we did has a small notational advantage. We
define a family $\hat{T} = \{\hat{T}_s\}_{s \in \cS}$ of operators
on $\cH_0$ as follows. First, we define
$\hat{T}_0$ to be the identity. Now assume that $s>0$. If $t\in \cS$ and $t \ngeq s$, then we define $\hat{T}_s (\delta_t \cdot \xi ) = 0$ for all
$\xi \in X(t) \otimes_{\sigma} H$ (or all $\xi \in H$, if $t=0$). If $\xi \in X(s) \otimes_\sigma H$, we define
$\hat{T}_s (\delta_s \cdot \xi ) = \delta_0 \cdot \tilde{T}_s \xi$. Finally, if $t
> s > 0$, we define
\be\label{eq:def:hat} \hat{T}_s \left(\delta_t \cdot (x_{t-s}
\otimes x_s \otimes h) \right) = \delta_{t-s} \cdot
\left(x_{t-s}\otimes \tilde{T}_s (x_s \otimes h) \right)
\ee
if $t
\geq s > 0$. Since $\tilde{T}_s$ is a contraction, $\hat{T}_s$
extends uniquely to a contraction in $B(\cH)$.

Let's stop to explain what we mean by equation (\ref{eq:def:hat}).
There are isomorphisms of correspondences $U_{t-s,s} : X(t-s)
\otimes X(s) \rightarrow X(t)$
Denote their inverses by
$U_{t-s,s}^{-1}$. When we write $x_{t-s} \otimes x_s$ for an
element of $X(t)$, we actually mean the image of this element by
$U_{t-s,s}$, and equation (\ref{eq:def:hat}) should be read as
\bes \hat{T}_s \left(\delta_t \cdot \left(U_{t-s,s}(x_{t-s}
\otimes x_s) \otimes h \right) \right) = \delta_{t-s} \cdot
\left(x_{t-s}\otimes \tilde{T}_s (x_s \otimes h) \right) ,\ees
or
\bes \hat{T}_s \left(\delta_t \cdot \left(\xi \otimes h \right) \right) = \delta_{t-s} \cdot
\left((I \otimes \tilde{T}_s) (U_{t-s,s}^{-1}\xi \otimes h) \right) .\ees
This shows that $\hat{T}$ is well defined.

We now show that $\hat{T}$ is a semigroup. Let $s,t,u \in \cS$. If
either $s = 0$ or $t = 0$ then it is clear that the semigroup
propety $\hat{T}_s \hat{T}_t = \hat{T}_{s+t}$ holds. Assume that
$s,t >0$. If $u \ngeq s+t$, then both $\hat{T}_{s} \hat{T}_{t}$
and $\hat{T}_{s + t}$ annihilate $\delta_u \cdot \xi$, for all
$\xi \in X(u) \otimes H$. Otherwise\footnote{Strictly speaking,
this only takes care of the case $u > s+t$ but the case $u = s+t$
is handled in a similar manner. This annoying issue will come up
again and again throughout the paper. Assuming that $\sigma$ is
unital, $X(0) \otimes H \cong H$, and one does not have to
separate the reasoning for the $X(s) \otimes H$ blocks and the $H$
blocks.},
\begin{align*}
\hat{T}_{s} \hat{T}_{t} \left(\delta_u (x_{u-s-t}\otimes x_s
\otimes x_t \otimes h ) \right) & = \hat{T}_{s} \left(\delta_{u-t}
(x_{u-s-t} \otimes x_s \otimes \tilde{T}_t(x_t \otimes h)) \right)
\\
& = \delta_{u-s-t} \left(x_{u-s-t} \otimes \tilde{T}_s(x_s \otimes \tilde{T}_t( x_t \otimes h)) \right) \\
& = \delta_{u-s-t} \left(x_{u-s-t} \otimes \tilde{T}_s(I
\otimes \tilde{T}_t )(x_s \otimes x_t \otimes h) \right) \\
& = \delta_{u-s-t} \left(x_{u-s-t} \otimes \tilde{T}_{s+t} (x_s
\otimes
x_t \otimes h) \right) \\
& = \hat{T}_{s+t} \left(\delta_{u} \left(x_{u-s-t} \otimes (x_s
\otimes x_t) \otimes h \right)\right) .
\end{align*}

We summarize the construction in the following proposition.
\begin{proposition}\label{prop:technology}
Let $A$, $X$, and $\cS$ and $(\sigma,T)$ be as above, and let
$$\cH = H \oplus \Big(\oplus_{0 \neq s \in \cS} X(s)\otimes_\sigma H \Big).$$
There exists a contractive semigroup $\hat{T} = \{\hat{T}_s\}_{s\in\cS}$ on $\cH$ such for all
$0\neq s\in\cS$, $x \in X(s)$ and $h\in H$,
$$\hat{T}_s \left(\delta_s \cdot x \otimes h \right) = T_s(x)h .$$
If $(\sigma,S)$ is another representation of $X$, and if $\hat{S}$ is the corresponding
contractive semigroup, then
$$\hat{T} = \hat{S}  \Rightarrow T = S .$$
\end{proposition}
One immediately sees a limitation in this construction: we cannot say that $\hat{T}$ is unique, or, equivalently,
that
$$\hat{T} = \hat{S}  \Leftrightarrow T = S .$$
For isometries the situation is better, if one puts several additional constraints on $\hat{T}$, but we shall not go into that.

\section{Regular isometric dilations of product systems}\label{sec:dil}
Let $H$ be a Hilbert space, and let $T = \{T_s\}_{s\in\cS}$ be a
semigroup of contractions over $\cS$. A semigroup $V =
\{V_s\}_{s\in\cS}$ on a Hilbert space $K \supseteq H$ is said to
be a \emph{regular dilation of $T$} if for all $s\in\cS - \cS$
$$P_H V_{s_-}^*V_{s_+} \big|_H = T_{s_-}^*T_{s_+} .$$
$V$ is said to be an \emph{isometric} dilation if it consists of
isometries . An isometric dilation $V$ is said to be a
\emph{minimal} isometric dilation if
$$K = \bigvee_{s\in\cS}V_s H .$$
In \cite{Shalit07b} we collected various results concerning
isometric dilations of semigroups, all of them direct consequences
of sections I.7 and I.9 in \cite{SzNF70}.

The notion of regular isometric dilations can be naturally
extended to representations of product systems.
\begin{definition}
Let $X$ be a product system over $\cS$, and let $(\sigma,T)$ be a
representation of $X$ on a Hilbert space $H$. An isometric
representation $(\rho,V)$ on a Hilbert space $K \supset H$ is said
to be a \emph{regular isometric dilation} if for all $a\in A =
X(0)$, $H$ reduces $\rho(a)$ and
$$\rho(a)\big|_H = \sigma(a)\big|_H ,$$
and for all $s\in\cS - \cS$
$$P_{X(s_-) \otimes H} \tilde{V}_{s_-}^* \tilde{V}_{s_+}\big|_{X(s_+) \otimes H} = \tilde{T}_{s_-}^* \tilde{T}_{s_+}.$$
Here, $P_{X(s_-) \otimes H}$ denotes the orthogonal projection of
${X(s_-) \otimes_{\rho} K}$ on ${X(s_-) \otimes_{\rho} H}$.
$(\rho,V)$ is said to be a \emph{minimal} dilation if
$$K = \bigvee \{V(x)h : x \in X, h \in H\} .$$
\end{definition}
In \cite{S06}, Solel studied regular isometric dilation of product
system representations over $\mathbb{N}^k$, and proved some
necessary and sufficient conditions for the existence of a regular
isometric dilation. One of our aims in this paper is to show how
the construction of Proposition \ref{prop:technology} can be used
to generalize \emph{some} of the results in \cite{S06}. The
following proposition is the main tool.

\begin{proposition}\label{prop:mainprop}
Let $A$ be a $C^*$-algebra, let $X = \{X(s)\}_{s \in \cS}$ be a
product system of $A$-correspondences over $\cS$, and let
$(T,\sigma)$ be a representation of $X$ on a Hilbert space $H$.
Let $\hat{T}$ and $\cH$ be as in Proposition
\ref{prop:technology}. Assume that $\hat{T}$ has a regular
isometric dilation. Then there exists a Hilbert space $K \supseteq
H$ and an isometric representation $V$ of $X$ on $K$, such that
\begin{enumerate}
    \item\label{it:V_0} $P_H$ commutes with $V_0 (A)$, and $V_0(a) P_H = \sigma(a) P_H$, for all $a \in
    A$;
    \item\label{it:regDil} $P_{X(s_-) \otimes H} \tilde{V}_{s_-}^* \tilde{V}_{s_+}\big|_{X(s_+) \otimes H} = \tilde{T}_{s_-}^* \tilde{T}_{s_+}$ for all $s \in \cS -
    \cS$;
    \item\label{it:min} $K = \bigvee \{V(x)h : x \in X, h\in H\} $
    ;
    \item\label{it:V*} $P_H V_s(x)\big|_{K \ominus H} = 0$ for all $s \in \cS$, $x \in X(s)$.
\end{enumerate}
That is, if $\hat{T}$ has a regular isometric dilation, then so
does $T$. If $\sigma$ is nondegenerate and $X$ is essential (that is, $A X(s)$ is dense in $X(s)$ for all $s\in \cS$)
then $V_0$ is also nondegenerate.
\end{proposition}
\begin{remark}
\emph{The results also hold in the $W^*$ setting, that is, if $A$ is a $W^*$-algebra,
$X$ is a product system of $W^*$-correspondences and $\sigma$ is normal, then $V_0$ is also normal.
A proof of this fact will appear in \cite{Shalit07a}.}
\end{remark}
\begin{proof}
Construct $\cH$ and $\hat{T}$ as in the previous section.

Let $\hat{V} = \{\hat{V}_s \}_{s \in \cS}$ be a minimal, regular,
isometric dilation of $\hat{T}$ on some Hilbert space $\cK$.
Minimality means that \bes \cK = \bigvee \{\hat{V}_t(\delta_s
\cdot(x \otimes h)) : s,t \in \cS, x \in X(s), h \in H \} . \ees

Introduce the Hilbert space $K$, \bes K = \bigvee
\{\hat{V}_s(\delta_s \cdot(x \otimes h)) : s \in \cS, x \in X(s),
h \in H \} . \ees

We consider $H$ as embedded in $K$ (or in $\cH$ or in $\cK$) by
the identification \bes h \leftrightarrow \delta_0 \cdot h . \ees

Next, we define a left action of $A$ on $\cH$ by \bes a \cdot
(\delta_s \cdot x \otimes h) = \delta_s \cdot ax \otimes h , \ees
for all $a \in A, s \in \cS - \{0\}, x \in X(s)$ and $h \in H$, and
\be\label{eq:V_0onH} a \cdot (\delta_0 \cdot h) = \delta_0 \cdot
\sigma(a) h \,\, , \,\, a \in A, h \in H . \ee

By Lemma 4.2 in \cite{cL94}, this extends to a bounded linear
operator on $\cH$. Indeed, this follows from the following
inequality:
\begin{align*}
\|\sum_{i=1}^n a x_i \otimes h_i \| & = \sum_{i,j=1}^n \langle
h_i,
T_0 (\langle a x_i, a x_j \rangle) h_j \rangle \\
& = \left \langle \big(T_0 (\langle a x_i, a x_j \rangle ) \big) (h_1, \ldots, h_n)^T,  (h_1, \ldots, h_n)^T \right \rangle_{H^{(n)}} \\
(*)& \leq \|a\|^2 \left \langle \big(T_0 (\langle x_i, x_j \rangle ) \big) (h_1, \ldots, h_n)^T,  (h_1, \ldots, h_n)^T \right \rangle_{H^{(n)}} \\
& = \|a\|^2 \|\sum_{i=1}^n x_i \otimes h_i \| .
\end{align*}
The inequality (*) follows from the complete positivity of $T_0$
and from $(\langle a x_i, a x_j \rangle ) \leq \|a\|^2 (\langle
x_i, x_j \rangle )$, which is the content of the cited lemma.

In fact, this is a $*$-representation (and it faithful if $T_0$
is). Explanation: it is clear that this is a homomorphism of
algebras. To see that it is a $*$-representation it is enough to
take $s \in \cS, x,y \in X(s)$ and $h,k \in H$ and to compute
\begin{align*}
\langle ax \otimes h, y \otimes k \rangle = \langle  h, T_0
(\langle ax,  y \rangle)  k \rangle = \langle  h, T_0 (\langle x,
a^* y \rangle)  k \rangle = \langle x \otimes h, a^*y \otimes k
\rangle ,
\end{align*}
(recall that the left action of $A$ on X(s) is adjointable). Note that
this left action commutes with $\hat{T}$: \bes a \hat{T}_s
(\delta_t x_{t-s} \otimes x_s \otimes h) = \delta_{t-s} a x_{t-s}
\otimes T_s (x_s) h = \hat{T}_s (\delta_t a x_{t-s} \otimes x_s
\otimes h) , \ees or \bes a \hat{T}_s (\delta_s x_s \otimes h) =
\delta_{0} \sigma (a)  T_s (x_s) h = \delta_{0}   T_s (a x_s) h=
\hat{T}_s (\delta_s a x_s \otimes h) . \ees

We shall now define a representation $V$ of $X$ on $K$. We wish to
define $V_0$ by the rules \be \label{eq:V_0 definition} V_0(a)
\hat{V}_s (\delta_s \cdot x_s \otimes h) = \hat{V}_s (\delta_s
\cdot a x_s \otimes h) , \ee and \bes V_0(a)  (\delta_0 \cdot h) =
\delta_0 \cdot \sigma(a) h . \ees
 To see that this extends
to a bounded, linear operator on $K$, let $\sum_{t} \hat{V}_t
(\delta_t \cdot x_t \otimes h_t) \in K$ (a finite sum), and
compute
\begin{align*}
\| \sum_{t} \hat{V}_t (\delta_t \cdot a x_t \otimes h_t) \|^2 &=
\sum_{s,t} \langle \hat{V}_s (\delta_s \cdot a x_s \otimes
h_s) , \hat{V}_t (\delta_t \cdot a x_t \otimes h_t) \rangle \\
&= \sum_{s,t} \langle \hat{V}_{(s-t)_-}^* \hat{V}_{(s-t)_+}
(\delta_s \cdot a x_s \otimes h_s) , \delta_t \cdot
a x_t \otimes  h_t \rangle \\
(*)&= \sum_{s,t} \langle \hat{T}_{(s-t)_-}^* \hat{T}_{(s-t)_+}
(\delta_s \cdot a x_s \otimes h_s) , \delta_t \cdot
a x_t \otimes  h_t \rangle \\
&= \sum_{s,t} \langle \hat{T}_{(s-t)_-}^* \hat{T}_{(s-t)_+}
(\delta_s \cdot a^* a x_s \otimes h_s) , \delta_t \cdot
 x_t \otimes  h_t \rangle \\
&= \sum_{s,t} \langle \hat{V}_s (\delta_s \cdot a^* a x_s \otimes
h_s) , \hat{V}_t (\delta_t \cdot  x_t \otimes h_t) \rangle .
\end{align*}
(The computation would have worked for finite sums including
summands from $H$, also). Step (*) is justified because $\hat{V}$
is a regular dilation of $\hat{T}$. This will be used repeatedly.
We conclude that if $a \in A$ is unitary then \bes \left \|
\sum_{t} \hat{V}_t (\delta_t \cdot a x_t \otimes h_t) \right \| =
\left \| \sum_{t} \hat{V}_t (\delta_t \cdot  x_t \otimes h_t)
\right \| . \ees For general $a \in A$, we may write $a =
\sum_{i=1}^4 \lambda_i u_i$, where $u_i$ is unitary and
$|\lambda_i| \leq 2 \|a \|$. Thus, \bes \left \| \sum_{t}
\hat{V}_t (\delta_t \cdot a x_t \otimes h_t) \right \| = \left
\|\sum_{i=1}^4  \lambda_i \sum_{t} \hat{V}_t (\delta_t u_i \cdot
x_t \otimes h_t) \right \| \leq 8 \|a\| \left \| \sum_{t}
\hat{V}_t (\delta_t \cdot  x_t \otimes h_t) \right \| . \ees In
fact, we will soon see that $V_0$ is a representation, so this
quite a lousy estimate. But we make it only to show that $V_0(a)$
can be extended to a well defined operator on $K$.

It is immediate that $V_0$ is linear and multiplicative. To see
that it is $*$-preserving, let $s,t \in \cS$, $x \in X(s), x' \in
X(t)$ and $h,h' \in H$.
\begin{align*}
\langle V_0(a)^* \hat{V}_s (\delta_s \cdot x \otimes h),
\hat{V}_t(\delta_t \cdot x' \otimes h') \rangle
& = \langle  \hat{V}_s (\delta_s \cdot x \otimes h), V_0(a) \hat{V}_t(\delta_t \cdot x' \otimes h') \rangle \\
& = \langle  \hat{V}_s (\delta_s \cdot x \otimes h),  \hat{V}_t(\delta_t \cdot ax' \otimes h') \rangle \\
& = \langle \hat{V}_{(s-t)_-}^* \hat{V}_{(s-t)_+} (\delta_s \cdot x \otimes h),  \delta_t \cdot ax' \otimes h' \rangle \\
& = \langle \hat{T}_{(s-t)_-}^* \hat{T}_{(s-t)_+} (\delta_s \cdot x \otimes h),  \delta_t \cdot ax' \otimes h' \rangle \\
& = \langle \hat{T}_{(s-t)_-}^* \hat{T}_{(s-t)_+} (\delta_s \cdot a^* x \otimes h), \delta_t \cdot x' \otimes h'\rangle\\
& = \langle \hat{V}_s (\delta_s \cdot a^* x \otimes h),  \hat{V}_t(\delta_t \cdot x' \otimes h') \rangle \\
& = \langle V_0(a^*) \hat{V}_s (\delta_s \cdot x \otimes h),
\hat{V}_t(\delta_t \cdot x' \otimes h') \rangle.
\end{align*}
Thus, $V_0(a)^* = V_0(a^*)$.

By (\ref{eq:V_0onH}), $H$ reduces $V_0 (A)$, and $V_0(a) \big|_H =
\sigma(a) \big|_H$ (under the appropriate identifications). The assertion about nondegeneracy of $V_0$ is clear from the definitions.

To define $V_s$ for $s > 0$, we will show that the rule \be
\label{eq:definition V_s} V_s(x_s) \hat{V}_t(\delta_t \cdot x_t
\otimes h) = \hat{V}_{s+t} (\delta_{s+t} \cdot x_s \otimes x_t
\otimes h) \ee can be extended to a well defined operator on $K$.
Let $\sum \hat{V}_{t_i}(\delta_{t_i} \cdot x_i \otimes h_i) $ be a
finite sum in $K$, and let $s \in \cS, x_s \in X(s)$. To estimate
\begin{align*}
\| \sum \hat{V}_{t_i + s}(\delta_{t_i + s} \cdot & x_s \otimes x_i
\otimes h_i) \|^2 = \\
&= \sum \langle \hat{V}_{t_i + s}(\delta_{t_i + s} \cdot x_s \otimes x_i \otimes h_i), \hat{V}_{t_j + s}(\delta_{t_j + s} \cdot x_s \otimes x_j \otimes h_j) \rangle \\
&= \sum \langle \hat{V}_s  \hat{V}_{t_i}(\delta_{t_i + s} \cdot x_s \otimes x_i \otimes h_i), \hat{V}_s \hat{V}_{t_j}(\delta_{t_j + s} \cdot x_s \otimes x_j \otimes h_j) \rangle \\
&= \sum \langle  \hat{V}_{t_i}(\delta_{t_i + s} \cdot x_s \otimes
x_i \otimes h_i), \hat{V}_{t_j}(\delta_{t_j + s} \cdot x_s \otimes
x_j \otimes h_j) \rangle,
\end{align*}
we look at each summand of the last equation. Denoting $\xi_i =
x_i \otimes h_i$, we have
\begin{align*}
\big\langle \hat{V}_{t_i}(\delta_{t_i + s} \cdot x_s \otimes
\xi_i), & \hat{V}_{t_j}(\delta_{t_j + s} \cdot x_s \otimes \xi_j)
\big\rangle =\\
&= \big\langle \hat{V}_{(t_i - t_j)_-}^* \hat{V}_{(t_i - t_j)_+}(\delta_{t_i + s} \cdot x_s \otimes \xi_i), \delta_{t_j + s} \cdot x_s \otimes \xi_j \big\rangle \\
&= \big\langle \hat{T}_{(t_i - t_j)_-}^* \hat{T}_{(t_i - t_j)_+}(\delta_{t_i + s} \cdot x_s \otimes \xi_i), \delta_{t_j + s} \cdot x_s \otimes \xi_j \big\rangle \\
&= \big\langle \delta_{t_j + s} \cdot x_s \otimes \left(I \otimes \tilde{T}_{(t_i-t_j)_-}^*\right) \left(I \otimes \tilde{T}_{(t_i-t_j)_+}\right) \xi_i, \\
& \quad \quad \quad \delta_{t_j + s} \cdot x_s \otimes \xi_j \big\rangle \\
&= \big\langle \delta_{t_j} \cdot \left(I \otimes \tilde{T}_{(t_i-t_j)_-}^*\right) \left(I \otimes \tilde{T}_{(t_i-t_j)_+}\right) \xi_i, \delta_{t_j} \cdot |x_s|^2  \xi_j \big\rangle \\
&= \big\langle \hat{T}_{(t_i - t_j)_-}^* \hat{T}_{(t_i - t_j)_+}(\delta_{t_i} \cdot \xi_i), \delta_{t_j} \cdot |x_s|^2 \xi_j \big\rangle \\
&= \big\langle \hat{V}_{t_i}(\delta_{t_i } \cdot |x_s| \xi_i), \hat{V}_{t_j}(\delta_{t_j} \cdot |x_s| \xi_j) \big\rangle \\
&= \big\langle V_0(|x_s|) \hat{V}_{t_i}(\delta_{t_i } \cdot
\xi_i), V_0(|x_s|) \hat{V}_{t_j}(\delta_{t_j} \cdot \xi_j)
\big\rangle ,
\end{align*}
(again, this argument works also if some $\xi$'s are in $H$). This
means that
\begin{align*}
\| \sum \hat{V}_{t_i + s}(\delta_{t_i + s} \cdot x_s \otimes x_i
\otimes h_i) \|^2
&= \| V_0 (|x_s|) \sum \hat{V}_{t_i}(\delta_{t_i} \cdot x_i \otimes h_i) \|^2 \\
&\leq \| V_0 (|x_s|)\|^2 \left \| \sum \hat{V}_{t_i}(\delta_{t_i}
\cdot x_i \otimes h_i) \right \|^2,
\end{align*}
so the mapping $V_s$ defined in (\ref{eq:definition V_s}) does
extend to a well defined operator on $K$. Now it is clear from the
definitions that for all $s \in \cS$, $(V_0, V_s)$ is a covariant
representation of $X(s)$ on $K$. We now show that it is isometric.
Let $s,t,u \in \cS$, $x, y \in X(s)$, $x_t \in X(t)$, $x_u \in
X(u)$ and $h,g \in H$. Then
\begin{align*}
\langle V_s(x)^* V_s(y) \hat{V}_t \delta_t \cdot x_t \otimes h,&
\hat{V}_u \delta_u \cdot x_u \otimes g \rangle =\\
&= \langle  \hat{V}_{t+s} \delta_{t+s} \cdot y \otimes x_t \otimes h, \hat{V}_{u+s} \delta_{u+s} \cdot x \otimes x_u \otimes g \rangle \\
&= \langle  \hat{V}_{(t-u)_-}^* \hat{V}_{(t-u)_+} \delta_{t+s} \cdot y \otimes x_t \otimes h,  \delta_{u+s} \cdot x \otimes x_u \otimes g \rangle \\
(*)&= \langle  \hat{V}_{(t-u)_-}^* \hat{V}_{(t-u)_+} \delta_t \cdot x_t \otimes h,  \delta_{u} \cdot \langle y, x \rangle x_u \otimes g \rangle \\
&= \langle  \hat{V}_{t} \delta_{t} \cdot x_t \otimes h, \hat{V}_{u} \delta_{u} \cdot \langle y, x \rangle x_u \otimes g \rangle \\
&= \langle V_0(\langle x,y \rangle) \hat{V}_t \delta_t \cdot x_t
\otimes h, \hat{V}_u \delta_u \cdot x_u \otimes g \rangle .
\end{align*}
The justification of (*) was carried essentially out in the proof
that $V_s (x_s)$ is well defined. Let us, for a change, show that
this computation works also for the case $u=0$:
\begin{align*}
\langle V_s(x)^* V_s(y) \hat{V}_t \delta_t \cdot x_t \otimes h,&
 \delta_0 \cdot g \rangle =\\
&= \langle  \hat{V}_{t+s} \delta_{t+s} \cdot y \otimes x_t \otimes h, \hat{V}_{s} \delta_{s} \cdot x \otimes g \rangle \\
&= \langle  \hat{V}_{t} \delta_{t+s} \cdot y \otimes x_t \otimes h,  \delta_{s} \cdot x \otimes g \rangle \\
&= \langle  \hat{T}_{t} \delta_{t+s} \cdot y \otimes x_t \otimes h,  \delta_{s} \cdot x \otimes g \rangle \\
&= \langle  \delta_{s} \cdot y \otimes T_t(x_t) \otimes h,  \delta_{s} \cdot x \otimes g \rangle \\
&= \langle  T_t(x_t) \otimes h, \sigma (\langle y, x\rangle) g \rangle \\
&= \langle  \hat{T}_t \delta_t \cdot x_t \otimes h, V_0 (\langle y, x\rangle) \delta_0 \cdot g \rangle \\
&= \langle  \hat{V}_t \delta_t \cdot x_t \otimes h, V_0 (\langle y, x\rangle) \delta_0 \cdot g \rangle \\
&= \langle V_0(\langle x,y \rangle) \hat{V}_t \delta_t \cdot x_t
\otimes h, \delta_0 \cdot g \rangle .
\end{align*}

We have constructed a family $V = \{V_s \}_{s \in \cS}$ of maps
such that $(V_0, V_s)$ is an isometric covariant representation of
$X(s)$ on $K$. To show that $V$ is a product system representation
of $X$, we need to show that the ``semigroup property" holds.

Let $h \in H$, $s,t,u \in \cS$, and let $x_s, x_t, x_u$ be in
$X(s), X(t), X(u)$, respectively. Then
\begin{align*}
V_{s+t} (x_s \otimes x_t) \hat{V}_u (\delta_u \cdot x_u \otimes h)
& = \hat{V}_{s+t+u}(\delta_{s+t+u} \cdot x_s \otimes x_t \otimes x_u \otimes h) \\
& = V_s (x_s) \hat{V}_{t+u}(\delta_{t+u} \cdot x_t \otimes x_u \otimes h) \\
& = V_s (x_s) V_t (x_t) \hat{V}_{u}(\delta_{u} \cdot x_u \otimes
h),
\end{align*}
so the semigroup property holds.

We have yet to show that $V$ is a minimal, regular dilation of
$T$. To see that it is a regular dilation, let $s \in \cS - \cS$,
$x_+ \in X(s_+), x_- \in X(s_-)$ and $h = \delta_0 \cdot h, g =
\delta_0 \cdot g \in H$. Using the fact that $\hat{V}$ is a
regular dilation of $\hat{T}$, we compute:
\begin{align*}
\langle \tilde{V}_{s_-}^* \tilde{V}_{s_+} (x_+ \otimes \delta_0
\cdot h),  (x_- \otimes \delta_0 \cdot g) \rangle &= \langle
\hat{V}_{s_+} (\delta_{s_+} x_+ \otimes h),
\hat{V}_{s_-} (\delta_{s_-} x_- \otimes g) \rangle \\
&= \langle \hat{V}_{s_-}^* \hat{V}_{s_+} (\delta_{s_+} x_+ \otimes
h),
\delta_{s_-} x_- \otimes g \rangle \\
&= \langle \hat{T}_{s_-}^* \hat{T}_{s_+} (\delta_{s_+} x_+ \otimes
h),
 \delta_{s_-} x_- \otimes g \rangle \\
&= \langle \tilde{T}_{s_+} (x_+ \otimes h),
\tilde{T}_{s_-} (x_- \otimes g) \rangle \\
&= \langle \tilde{T}_{s_-}^* \tilde{T}_{s+} (x_+ \otimes h),
 x_- \otimes g \rangle .
\end{align*}

$V$ is a minimal dilation of $T$, because
\begin{align*}
K &= \bigvee \{\hat{V}_s(\delta_s \cdot(x \otimes h)) : s \in \cS,
x \in X(s),
h \in H \} \\
&= \bigvee \{V_s(x) (\delta_0 \cdot h) : s \in \cS, x \in X(s), h
\in H \} .
\end{align*}

Finally, let us note that item \ref{it:V*} from the statement of the proposition is true for any
minimal isometric dilation (of any c.c. representation of a
product system over any semigroup). Indeed, let $V$ be a minimal
isometric dilation of $T$ on $K$. Let $x_s \in X(s), x_t \in X(t)$
and $h \in H$. Then
\begin{align*}
P_H V_s(x_s) V_t(x_t) h & = P_H V_{s+t} (x_s \otimes x_t)h \\
& = T_{s+t} (x_s \otimes x_t)h = T_s(x_s) T_t(x_t) h \\
& = P_H V_s(x_s) P_H V_t(x_t) h.
\end{align*}
But $K = \bigvee \{V_s(x)h : s \in \cS, x \in X(s), h \in H \}$,
so $P_H V_s(x_s) P_H = P_H V_s(x_s)$, from which item
(\ref{it:V*}) follows.
\end{proof}

It is worth noting that, as commensurable semigroups are
countable, if $\cS = \sum_{i=1}^\infty \cS_i$, then, using the
notation of the above proposition, separability of $H$ implies
that $K$ is separable. It is also worth recording the following
result, the proof of which essentially appears in the proof of
Proposition 3.7, \cite{S06}.

\begin{proposition}\label{prop:unique}
Let $X$ be a product system over $\cS$, and let $T$ be a
representation of $X$. A minimal, regular, isometric dilation of
$T$ is unique up to unitary equivalence.
\end{proposition}

\section{Regular isometric dilations of doubly commuting representations}\label{sec:dbly}
It is well known that in order that a $k$-tuple $(T_1, T_2,
\ldots, T_k)$ of contractions have a commuting isometric dilation,
it is not enough to assume that the contractions commute. One of
the simplest sufficient conditions that one can impose on $(T_1,
T_2, \ldots, T_k)$ is that it \emph{doubly commute}, that is
$$T_j T_k = T_k T_j \,\, {\rm and} \,\, T_j^* T_k= T_k T_j^* $$
for all $j \neq k$. Under this assumption, the $k$-tuple $(T_1,
T_2, \ldots, T_k)$ actually has regular unitary dilation. In fact,
if the $k$-tuple $(T_1, T_2, \ldots, T_k)$ doubly commutes then it
also has a \emph{doubly commuting} regular \emph{isometric}
dilation (see Proposition 3.5 in \cite{Shalit07b} for the simple
explanation). This fruitful notion of double commutation can be
generalized to representations as follows.
\begin{definition}
A representation $(\sigma,T)$ of a product system $X$ over $\cS$
is said to \emph{doubly commute} if \bes (I_{{\bf e_k}(s_k)}
\otimes \tilde{T}_{{\bf e_j}(s_j)})
  (t \otimes I_H) (I_{{\bf e_j}(s_j)} \otimes \tilde{T}_{{\bf e_k}(s_k)}^*) = \tilde{T}_{{\bf e_k}(s_k)}^* \tilde{T}_{{\bf e_j}(s_k)}
\ees
for all $j \neq k$ and all nonzero $s_j\in \cS_j, s_k\in \cS_k$, where
$t$ stands for the isomorphism between $X({\bf e_j}(s_j)) \otimes
X({\bf e_k}(s_k))$ and $X({\bf e_k}(s_k)) \otimes X({\bf
e_j}(s_j))$, and $I_{s}$ is shorthand for $I_{X(s)}$.
\end{definition}
\begin{theorem}\label{thm:dbly}
Let $A$ be a $C^*$-algebra, let $X = \{X(s)\}_{s \in \cS}$ be a
product system of $A$-correspondences over $\cS$, and let
$(\sigma,T)$ be doubly commuting representation of $X$ on a
Hilbert space $H$. There exists a Hilbert space $K \supseteq H$
and a minimal, doubly commuting, regular isometric representation
$V$ of $X$ on $K$.
\end{theorem}
\begin{proof}
Construct $\cH$ and $\hat{T}$ as in section \ref{sec:rep}.

We now show that $\hat{T}_{{\bf e_j}(s_j)}$ and $\hat{T}_{{\bf
e_k}(s_k)}$ doubly commute for all $j \neq k$, and all $s_j\in
\cS_j, s_k\in \cS_k$. Let $t \in \cS$, $x \in X(t), y \in X({\bf
e_j}(s_j))$ and $h \in H$. Using the assumption that $T$ is a
doubly commuting representation,
\begin{align*}
\hat{T}_{{\bf e_k}(s_k)}^* \hat{T}_{{\bf e_j}(s_j)} (\delta_{t+{\bf e_j}(s_j)} \cdot x \otimes y \otimes h)
&= \hat{T}_{{\bf e_k}(s_k)}^* \left(\delta_{t} \cdot x \otimes \tilde{T}_{{\bf e_j}(s_j)} (y \otimes h) \right) \\
&= \delta_{t+{\bf e_k}(s_k)} \cdot x \otimes \tilde{T}_{{\bf e_k}(s_k)}^* \tilde{T}_{{\bf e_j}(s_j)} (y \otimes h) \\
&= \delta_{t+{\bf e_k}(s_k)} \cdot x \otimes \left( (I_{{\bf e_k}(s_k)} \otimes \tilde{T}_{{\bf e_j}(s_j)})
  (t \otimes I_H) (I_{{\bf e_j}(s_j)} \otimes \tilde{T}_{{\bf e_k}(s_k)}^*)
  (y \otimes h) \right) \\
&= \hat{T}_{{\bf e_j}(s_j)} \hat{T}_{{\bf e_k}(s_j)}^* (\delta_{t+{\bf e_j}(s_j)} \cdot x \otimes y \otimes h) ,
\end{align*}
where we have written $t$ for the
isomorphism between $X({\bf e_j}(s_j)) \otimes X({\bf e_k}(s_k))$ and
$X({\bf e_k}(s_k)) \otimes X({\bf e_j}(s_j))$, and we haven't written the
isomorphisms between $X(s) \otimes X(t)$ and $X(s+t)$.

By Corollary 3.7 in \cite{Shalit07b}\footnote{We have to mention
that the proof of Corollary 3.7, \cite{Shalit07b}, is based on
Theorem 3.10 of \cite{S06}. This may seem like an awkward
situation since we are trying to promote a \emph{new} method of
analyzing representations. Of course, Theorem 3.10 of \cite{S06}
could have been proved in the setting of contraction semigroups on
Hilbert spaces, so there is no real departure from our model.},
there exists a minimal, regular isometric dilation $\hat{V} =
\{\hat{V}_s \}_{s \in \cS}$ of $\hat{T}$ on some Hilbert space
$\cK$, such that $\hat{V}_{{\bf e_j}(s_j)}$ and $\hat{V}_{{\bf
e_k}(s_k)}$ doubly commute for all $j\neq k,s_j\in \cS_j, s_k \in
\cS_k$. The construction in Proposition \ref{prop:mainprop} gives
rise to a minimal, regular isometric dilation $V$ of $T$ on some
Hilbert space $K$.

To see that $V$ is doubly commuting, one computes what one should
using the fact that $\hat{V}$ is a minimal, doubly commuting,
regular isometric dilation of $\hat{T}$ (all the five adjectives
attached to $\hat{V}$ play a part). This takes about 4 pages of
handwritten computations, so is omitted. Let us indicate how it is
done. For any $i\in\Omega$, $s_i \in \cS_i$, write $\tilde{V}_i$
for $\tilde{V}_{X({\bf e_i}(s_i))}$, $I_i$ for $I_{X({\bf
e_i}(s_i))}$, and so on. Taking $j\neq k$, $s_j \in \cS_j, s_k \in
\cS_k$, operate with
$$\tilde{V}_k (I_k\otimes \tilde{V}_j)(t_{j,k}\otimes I_J)(I_j \otimes \tilde{V}_k^*)$$
and with
$$\tilde{V}_k \tilde{V}_k^* \tilde{V}_j $$
on a typical element of $X({\bf e_j}(s_j)) \otimes  K$ of the
form: \be\label{eq:element} x \otimes \hat{V}_s (\delta_s \cdot
x_s \otimes h) , \ee to see that what you get is the same. One has
to separate the cases where ${\bf e_k}(s_k) \leq s$ and ${\bf
e_k}(s_k) \nleq s$ (this is the case where the fact that $\hat{V}$
is a doubly commuting semigroup comes in). Because $\tilde{V}_k$
is an isometry, and the elements (\ref{eq:element}) span $X({\bf
e_j}(s_j)) \otimes  K$, one has
$$\tilde{V}_k^* \tilde{V}_j =  (I_k\otimes \tilde{V}_j)(t_{j,k}\otimes I_J)(I_j \otimes \tilde{V}_k^*) .$$
That will conclude the proof.
\end{proof}

\section{A sufficient condition for the existence of a regular isometric dilation}\label{sec:further}
Using the above methods, one can, quite easily, arrive at the
following result, which is, for the case $\cS = \mathbb{N}^k$, one half of Theorem 3.5 of \cite{S06}.
\begin{theorem}\label{thm:reg}
Let $X$ be a product system over $\cS$, and let
$T$ be a representation of $X$. If
\be\label{eq:NS}
\sum_{u\subseteq v}(-1)^{|u|}\left(I_{s[v]-s[u]} \otimes \tilde{T}^*_{s[u]}\tilde{T}_{s[u]} \right) \geq 0
\ee
for all finite subsets $v\subseteq \Omega$ and all $s\in\cS$, then $T$ has a regular isometric dilation.
\end{theorem}
\begin{proof}
Here are the main lines of the proof. Construct $\hat{T}$ as in section \ref{sec:rep}. From (\ref{eq:NS}), it follows that
$\hat{T}$ satisfies
$$\sum_{u\subseteq v}(-1)^{|u|}\hat{T}^*_{s[u]}\hat{T}_{s[u]} \geq 0 ,$$
for all finite subsets $v\subseteq \Omega$ and all $s\in\cS$,
which, by Proposition 3.5 and Theorem 3.6 in \cite{Shalit07b}, is
a necessary and sufficient condition for the existence of a
regular isometric dilation $\hat{V}$ of $\hat{T}$. The result now
follows from Proposition \ref{prop:mainprop}.
\end{proof}

Among other reasons, this example has been put forward to illustrate the limitations of our method. By Theorem
3.5 of \cite{S06}, when $\cS = \mathbb{N}^k$, equation (\ref{eq:NS}) is a \emph{necessary}, as well as a sufficient, condition that $T$ has a regular isometric dilation. But our contstruction ``works only in one direction", so are able to prove only sufficient conditions (roughly speaking). We believe that, using the methods of \cite{S06} combined with commensurability considerations, one would be able to show that (\ref{eq:NS}) is indeed a necessary condition for the existence of a regular isometric dilation (over $\cS$).

\section{Acknowledgements}
The author is supported by the Jacobs School of Graduate Studies
and the Department of Mathematics at the Technion - I.I.T, and by
the Gutwirth Fellowship. This research is part of the author's
PhD. thesis, done under the supervision of, and with plenty of
from, Professor Baruch Solel.

\bibliographystyle{amsplain}

\end{document}